\documentclass[12pt]{amsart}

\def\@typesizes{%
       \or{5}{6.5}\or{6}{7.5}\or{7}{8.5}\or{8}{11}\or{9}{12}%
       \or{10}{13}
       \or{\@xipt}{14}\or{\@xiipt}{15}\or{\@xivpt}{18}%
       \or{\@xviipt}{20}\or{\@xxpt}{24}}

\usepackage{amsthm}
\usepackage{amsmath}
\usepackage{amssymb}
\usepackage{graphicx}
\usepackage{enumerate}
\allowdisplaybreaks
\numberwithin{equation}{section}

\theoremstyle{plain}
\newtheorem{theorem}{ Theorem}[section]
\newtheorem{proposition}[theorem]{ Proposition}
\newtheorem{lemma}[theorem]{ Lemma}
\newtheorem{corollary}[theorem]{ Corollary}
\newtheorem{example}[theorem]{ Example}
\newtheorem{remark}[theorem]{ Remark}
\newtheorem{definition}[theorem]{ Definition}
\newtheorem{conjecture}{ Conjecture}

\def\BET{\begin{theorem}}
\def\ENT{\end{theorem}}
\def\BEP{\begin{proposition}}
\def\ENP{\end{proposition}}
\def\BEL{\begin{lemma}}
\def\ENL{\end{lemma}}
\def\BEC{\begin{corollary}}
\def\ENC{\end{corollary}}
\def\BEE{\begin{example} \rm}
\def\ENE{\end{example}}
\def\BER{\begin{remark} \rm}
\def\ENR{\end{remark}}
\def\BED{\begin{definition} \rm}
\def\END{\end{definition}}
\def\BECJ{\begin{conjecture}}
\def\ENCJ{\end{conjecture}}

\def\bea{\begin{eqnarray}}
\def\eea{\end{eqnarray}}

\def\beq{\begin{equation}}
\def\eeq{\end{equation}}

\def\beal{\begin{align*}}

\def\eeal{ \end{align*} }

\newcommand{\Endproof}{\ifmmode\eqno\square
  \else{\hfill$\square$\parfillskip=0pt\par}\fi}

\makeatletter
\newif\ifmeanright
\newcommand{\finite}[1]{\meanrightfalse
\ifx#1H\meanrighttrue\fi
\ifx#1B\meanrighttrue\fi
\ifmeanright{\lefteqn{\overset\circ{\hphantom{#1'}\vphantom{\prime}}}#1}
\else{\lefteqn{\overset\circ{\hphantom{#1}\vphantom{\prime}}}#1}\fi}
\makeatother

\begin{document}

\centerline{\bf Examples of Plentiful Discrete Spectra in}
\centerline{\bf Infinite Spatial Cruciform Quantum  Waveguides}

\bigskip

\centerline{\it Fedor Bakharev\footnote{Mathematics and Mechanics Faculty, St. Petersburg State University,
7/9 Universitetskaya nab., St. Petersburg, 199034 Russia, fbakharev@yandex.ru}, Sergey Matveenko\footnote{Chebyshev Laboratory, St. Petersburg State University, 14th Line V.O., 29B, Saint Petersburg 199178 Russia;
National Research University Higher School of Economics, 194100, Saint-Petersburg, Kantemirovskaya st., 3A, office 417, {matveis239@gmail.com}} 
and Sergey Nazarov\footnote{Mathematics and Mechanics Faculty, St. Petersburg State University,
7/9 Universitetskaya nab., St. Petersburg, 199034 Russia; Saint-Petersburg State Polytechnical
University, Polytechnicheskaya ul., 29, St. Petersburg, 195251, Russia; Institute of
Problems of Mechanical Engineering RAS, V.O., Bolshoj pr., 61, St. Petersburg, 199178,
{srgnazarov@yahoo.co.uk}}}

\bigskip

{\bf Abstract.}
{Spatial cruciform quantum waveguides (the Dirichlet problem for Laplace operator)
are constructed such that the total multiplicity of the discrete spectrum
exceeds any preassigned number.
}


\bigskip

\section{Introduction}
In the informative paper \cite{Gri} D. Grieser proved in particular that the spectrum of a finite
lattice of thin (with diameter $O(\varepsilon)$, $\varepsilon \ll 1$) quantum waveguides gets much more complicated
asymptotic structure than the Neumann Laplacian in the same thin lattice whose spectrum is described
by the classical L. Pauling model \cite{Pau}, that is, a one-dimensional skeletal lattice graph
with differential structures on edges and the classical Kirchhoff transmission conditions at the vertices.
Indeed, according to \cite{Gri}, the low-frequency range of the Dirichlet Laplacian modeling quantum waveguides,
consists of the finite family
\begin{equation}
\label{0}
\varepsilon^{-2}\lambda_1, \ldots, \varepsilon^{-2} \lambda_J
\end{equation}
where $\lambda_1$, \ldots, $\lambda_J$ are eigenvalues in the discrete spectra of the Dirichlet problem in unbounded
domains which describe the boundary layer phenomenon, have several cylindrical
outlets to infinity and are obtained from
the lattice by stretching the coordinate systems centered at the nodes. For the Neumann case,
set \eqref{0} is obviously empty but in the Dirichlet case to detect those eigenvalues
becomes a challenging question because variational methods do not apply. In our paper we will demonstrate
particular shapes of cruciform spatial quantum waveguides where the number $J$ in list \eqref{0}
can be made arbitrarily large.

For two-dimensional rectangular lattices of thin (with width $\varepsilon \ll 1$) quantum waveguides
in Fig. \ref{fig1},a, comprehensive results have been obtained in \cite{na480, na548, naIAN} where
in particular it was proved that the discrete spectrum of the cruciform waveguide
$\Pi=\{(x_1,x_2)\in \mathbb{R}^2 : |x_1|<1/2 \mbox{ or } |x_2|<1/2\}$ consists of the only eigenvalue
$\Lambda_\Pi\in (0,\lambda_\dagger(\Pi))$ while clearly the continuous spectrum $[\lambda_\dagger(\Pi),+\infty)$
has the cut-off value $\lambda_\dagger(\Pi)=\pi^2$. The corresponding eigenfunction normalized in $L^2(\Pi)$
will be denoted by $U_\Pi$:
\begin{equation}
\label{DPi}
-\Delta_y U_{\Pi}(y)=\Lambda_\Pi U_\Pi(y),\quad y\in\Pi, \quad U_\Pi=0,\quad y\in \partial\Pi.
\end{equation}
Moreover, this homogeneous  Dirichlet problem
with the threshold parameter $\lambda=\lambda_\dagger(\Pi)$ has no bounded solutions, neither eigenfunction
decaying at infinity, nor solution which stabilizes to $c_j^\pm \cos x_j$ as $x_{3-j}\to \pm \infty$, $j=1,2$.
The latter, again due to a result in \cite{Gri}, implies that the mid-frequency range of the spectrum
of the rectangular lattice of thin planar quantum waveguides is described by ordinary differential equations on edges
of the rectangular graph, Fig \ref{fig1},b, but in contrast to the Pauling model \cite{Pau}, all vertices are supplied with
the Dirichlet conditions splitting
the graph into independent line intervals.

Quite the same conclusions but by means of a substantively modified approach
were made in \cite{BaMaNaDan, BaMaNaAa} for a rectangular quantum lattice composed of thin circular cylinders
as well as for the cruciform waveguide $Q=\{(x_1,x_2,x_3): x_1^2+x_3^2<1/4 \mbox{ or } x_2^2+x_3^2<1/4\}\subset\mathbb{R}^3$, Fig. \ref{fig2},a,
which also has only one point in the discrete spectrum of the Dirichlet problem. However, three-dimensional geometry offers much
many options and in the sequel we will describe cross-sections of cylinders in the cruciform junction that provide any
prescribed number $J$ in \eqref{0}.
We mention that, after factoring $\cos (\pi x_3)$ out, the spatial waveguide with the unit square cross-section
in Fig. \ref{fig2},b, inherits the only isolated eigenvalue $\Lambda_\Pi<\pi^2$ from the planar waveguide $\Pi$ but the total
multiplicity of the discrete spectrum in the waveguide with the right-angled rhombic cross-section in Fig. \ref{fig2},c, is not known yet.

\begin{figure}[ht]
\begin{center}
\begin{minipage}[ht]{0.3\linewidth}
\includegraphics[width=0.85\linewidth]{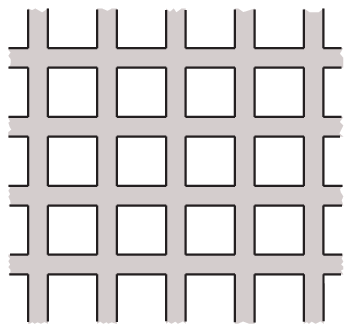}
\centerline{a)}
\end{minipage}
\begin{minipage}[ht]{0.3\linewidth}
\includegraphics[width=0.85\linewidth]{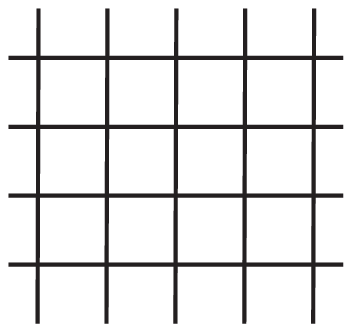}
\centerline{b)}
\end{minipage}
\begin{minipage}[ht]{0.3\linewidth}
\includegraphics[width=0.85\linewidth]{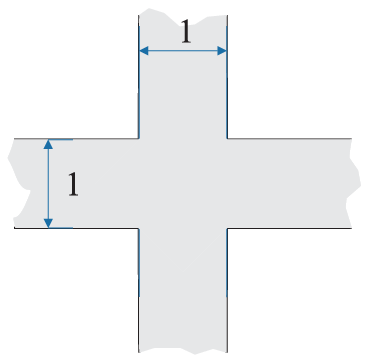}
\centerline{c)}
\end{minipage}
\end{center}
\caption{Thin rectangular lattice (a), its one-dimensional model (b) and the planar cruciform waveguide (c).}
 \label{fig1}
\end{figure}

\begin{figure}[ht]
\begin{minipage}[ht]{0.3\linewidth}
\includegraphics[width=0.85\linewidth]{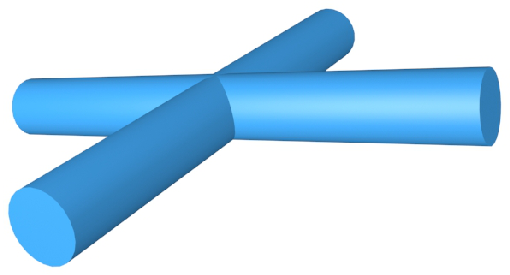}
\end{minipage}
\begin{minipage}[ht]{0.3\linewidth}
\includegraphics[width=0.85\linewidth]{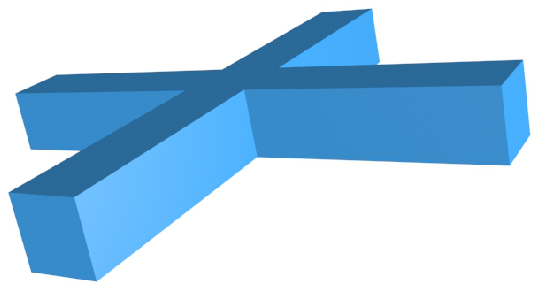}
\end{minipage}
\begin{minipage}[ht]{0.3\linewidth}
\includegraphics[width=0.85\linewidth]{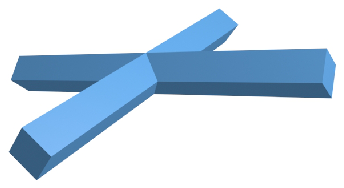}
\end{minipage}
\begin{minipage}[ht]{0.3\linewidth}
\centerline{a)}
\end{minipage}
\begin{minipage}[ht]{0.3\linewidth}
\centerline{b)}
\end{minipage}
\begin{minipage}[ht]{0.3\linewidth}
\centerline{c)}
\end{minipage}
\caption{Spatial cruciform waveguides, circular (a), square (b) and rhombic (c).}
 \label{fig2}
\end{figure}

\section{Statement of the problem}
Let $Q^{H}=Q_1^H\cup Q_2^H$ be a union of the cylinders
\begin{equation}
\label{WGD}
 Q_j^H=\Big\{x=( x_1, x_2, x_3)\in \mathbb{R}^3
\colon  (x_{3-j},H^{-1}x_3)\in \omega\Big\},
\end{equation}
where $H>0$ is a parameter and $\omega\subset \mathbb{R}^2$ is a domain which is enveloped by a Lipschitz contour  $\partial \omega$
and contains the origin $O=(0,0)$. We consider two particular cases:

({\bf i}) the diamond-shaped cross-section with the unscaled right-angled rhombic prototype 
$$
\omega^\lozenge=\Big\{(y,z)\in\mathbb{R}^2\colon|y|+|z|<1/2\Big\},
$$

({\bf ii}) the ellipsoidal cross-section with the unscaled circular prototype
$$
\omega^\circ=\Big\{(y,z)\in\mathbb{R}^2\colon|y|^2+|z|^2<1/4\Big\}.
$$

\begin{figure}[ht]
\begin{minipage}[ht]{0.4\linewidth}
\includegraphics[width=\linewidth]{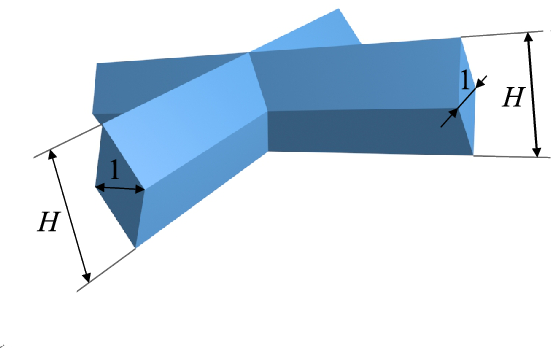}
\end{minipage}
\begin{minipage}[ht]{0.4\linewidth}
\includegraphics[width=\linewidth]{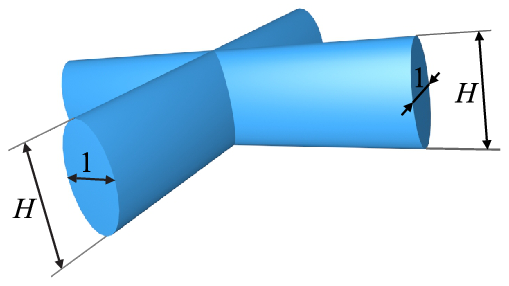}
\end{minipage}\qquad
\begin{minipage}[ht]{0.4\linewidth}
\centerline{a)}
\end{minipage}\qquad
\begin{minipage}[ht]{0.4\linewidth}
\centerline{b)}
\end{minipage}
\caption{Cruciform waveguides with high ($H\gg 1$) rhombic (a) and ellipsoid (b) cross-section of cylinders }
\end{figure}

We are able to treat the case of big $H$, that is $H\to+\infty$,
while a desirable information when $H\to+0$ is still unattainable for us.

We consider the Dirichlet spectral problem for the Laplace operator
\begin{equation} \label{DP}
 -\Delta_x u(x)=\lambda u(x), \,\,  x\in Q^{H} ,\quad
u( x)=0,\,\,  x\in \partial Q^{H}\,.
\end{equation}
Its variational formulation reads:
\begin{equation}
\label{Var}
(\nabla_x u,\nabla_x v)_{H}=\lambda (u,v)_H,\quad v\in H^1_0(Q^H),
\end{equation}
where $(\cdot,\cdot)_H$ is the natural inner product in the Lebesgue space $L^2(Q^H)$
and $H^1_0(Q^H)$ is a subspace of functions in the Sobolev space $H^1(Q^H)$ vanishing
at the boundary. Since the bilinear form on the left of \eqref{Var}
is closed and positive definite, the formulation gives rise to unbounded
positive definite and self-adjoint operator $\mathcal{A}^H$ in the Hilbert space $L^2(Q^H)$.

It is known that the continuous spectrum $\sigma_c^H$ of the operator $\mathcal{A}^H$ is the semi-axis $[\lambda_\dagger^H,+\infty)$
where the principal threshold $\lambda_\dagger^H:=\lambda_\dagger(Q^H)>0$ coincides with the first eigenvalue of the
Dirichlet problem on the cross-section in \eqref{WGD}
$$
\omega^H=\{(y,z)\in\mathbb{R}^2:(y,H^{-1}z)\in \omega\}.
$$

According to \cite{Freitas, na539} and \cite{BorFre1, NaPeTa} the threshold value admits the asymptotic form
\begin{equation}
\label{cutoff}
\lambda_\dagger^H=\pi^2+\mu_\dagger H^{-\alpha} +O(H^{-2\alpha}),\quad H\to+\infty,
\end{equation}
where $\alpha=\alpha^\lozenge=2/3$ and $\alpha=\alpha^\circ=1/2$. Moreover $\mu_\dagger$
is the smallest eigenvalue of either the Airy equation
\begin{equation}
\label{airy}
-\partial^2_\zeta w(\zeta)+4\pi^2|\zeta| w(\zeta) =\nu w(\zeta), \quad \zeta\in\mathbb{R}\,
\end{equation}
in case ({\bf i}), or the harmonic oscillator equation
\begin{equation}
\label{haros}
-\partial^2_\zeta w(\zeta)+4\pi^2\zeta^2 w(\zeta) =\nu w(\zeta), \quad \zeta\in\mathbb{R}\,
\end{equation}
in case ({\bf ii}).

The total multiplicity $\# \sigma_d^H$ of the discrete spectrum \eqref{0}
$$
0<\lambda_1^H< \lambda_2^H\leq\ldots \leq \lambda_{\#\sigma_d^H}^H<\lambda_\dagger^H,
$$
of the operator $\mathcal{A}^H$ is finite and $\#\sigma_d^H\geq 1$ according to \cite{na480}.

The main result of the paper implies the following assertion for sufficiently high diamonds and ellipses.

\begin{theorem}
\label{Main}
For any $N\in \mathbb{N}=\{1,2,3,\ldots\}$, one finds positive $H_N$ and  $C_N$ such that
the inequality $H>H_N$ guarantees  that $\# \sigma^H_d \geq N$ and
\[\lambda_N^H< \Lambda_\Pi+C_N H^{-\alpha}<\pi^2<\lambda_\dagger^H,\]
where $\alpha=\alpha^\lozenge=2/3$ and $\alpha=\alpha^\circ=1/2$.
\end{theorem}

Our scheme of the proof works for both cases ({\bf i}) and ({\bf ii})
in a very similar manner. Namely, based on the asymptotic analysis performed in the papers \cite{na259, FredSol, Freitas, BorFre1, na539, NaPeTa}
and others, we use basic properties of eigenfunctions of the ordinary differential equations \eqref{airy} and \eqref{haros} with 
$\lambda_\dagger(\Pi)=\pi^2$
replaced by $\Lambda_\Pi$, prepare appropriate, in particular ``almost orthonormalized'', trial functions and finally apply
the max-min principle for the operator $\mathcal{A}^H$ to detect its eigenvalues below the cutoff value \eqref{cutoff}. The
conclusive observation demonstrates that the number of constructed linear independent trial functions grows infinitely when $H\to+\infty$.
This similarity and common fundamental properties of eigenfunctions in \eqref{airy} and \eqref{haros} as well as recalling the above-mentioned
analyses allow us to focus on a bit more complicated case ({\bf i}) and only to outline some features of case ({\bf ii}).

\section{Total multiplicity in the diamond case}
Our construction of announced trial functions $\Phi_j^H$ in the max-min principle, see \eqref{mm} below, imitates the asymptotic procedures
proposed in the papers \cite{Freitas, na539} where it was proved that eigenfunctions of the Dirichlet problem in thin polygons and polyhedra are
localized in the vicinity of certain vertices and/or edges while in the waveguide $Q^H$ with $H>1$ the role of such concentrator
is taken by four broken edges in $\{x\in \partial Q^H: z=x_3=0\}$, that is, the boundary of the planar waveguide $\Pi$ of width 1, the maximal one
among all cross-sections of the polyhedron $Q^H$,
\begin{equation}
\label{PiT}
\Pi_\tau=\{y : (y,\tau)\in Q^H\}  \mbox{ of width } h(\tau)=1-\frac{2}{H}|\tau|, \quad \tau\in [-\frac{H}{2},\frac{H}{2}].
\end{equation}

We accept the ansatz
\begin{equation}
\label{phiH}
\Phi_n^H(x)=H^{-\alpha/4}\chi_H(\zeta)w_n(\zeta)U_\Pi((h(z)^{-1}y), \quad n\geq1,
\end{equation}
where $y=(x_1,x_2)$ are the ``horizontal'' coordinates and
\begin{equation}
\label{zetaH}
\zeta=H^{-\alpha/2}z
\end{equation}
is the compressed ``vertical'' coordinate with the exponent $\alpha=\alpha^\lozenge=2/3$. Furthermore,
$\chi_H(\zeta)=\chi(2H^{-\alpha}\zeta)$,
$\chi\in C^\infty(\mathbb{R})$ is a cut-off function,
\begin{equation}
\label{chi}
 \chi(\tau)=1 \mbox{ for } |\tau|<1/4\quad \mbox{and} \quad \chi(\tau)=0 \mbox{ for } |\tau|>1/2,
\end{equation}
and the last multiplier in \eqref{phiH} involves the
eigenfunction $U_\Pi$ in $\Pi$, so that it is defined on the cross-section $\Pi_z$, see \eqref{PiT}.
We emphasize that the function $\Phi_n^H$ vanishes on the boundary $\partial Q^H$ because the multiplier
$U^H(x)=U_\Pi(h(z)^{-1}y)$ is taken from the Dirichlet problem \eqref{DPi}. Finally, the ``approximate'' eigenvalue takes the form
\begin{equation}
\label{LamH}
\Lambda_\Pi+\mu_n H^{-\alpha},
\end{equation}
while $\{\mu_n, w_n\}$ in \eqref{LamH} and \eqref{phiH} is an eigenpair of the Airy equation
\begin{equation}
\label{4}
-\partial^2_\zeta w(\zeta)+4\Lambda_\Pi |\zeta| w(\zeta)=\mu w(\zeta), \quad \zeta\in\mathbb{R},
\end{equation}
in the class of functions decaying at infinity. It is known,
cf. \cite[Th. XIII.67]{RS4}, that the spectrum of the equation \eqref{4} is discrete and composes the monotone unbounded sequence of
simple eigenvalues
$$
0<\mu_1<\mu_2<\ldots<\mu_n<\ldots\to+\infty,
$$
while the corresponding eigenfunctions $w_1$, $w_2$, \ldots, $w_n$, \ldots can be subject
to the normalization and orthogonality conditions
\begin{equation}
\label{4.5}
(w_j,w_k)_{L^2(\mathbb{R})}=\delta_{j,k},\quad j,k\in\mathbb{N},
\end{equation}
where $\delta_{j,k}$ stands for the Kronecker symbol.
Moreover, $w_k$ can be expressed in terms of the Airy function,
\begin{equation}
\label{wnzet}
w_n(\zeta)=a_n \mathrm{Ai}((4\Lambda_\Pi)^{1/3}|\zeta|-(4\Lambda_\Pi)^{-2/3}\mu_n), \quad
\end{equation}
where $a_n$ is a normalization factor.
According to \cite[10.4.59]{AS} (see, e.g., the book \cite{BB} for a proof),
\begin{equation}
\label{AiT}
\mathrm{Ai}(t)\sim \frac{1}{2}\pi^{-1/2}t^{-1/4}e^{-T}\sum_0^{+\infty}
(-1)^kc_k T^{-k},\quad t>0,
\end{equation}
where
\[T=\frac{2}{3}t^{3/2},\;c_0
=1,\;c_k=\frac{\Gamma\left(3k+1/2\right)}
{54^kk!\Gamma\left(k+1/2\right)},
\]
and $ \Gamma $ denotes the gamma function.
In view of the exponential decay of \eqref{wnzet} caused by \eqref{AiT},
the following inequality becomes evident.

\begin{lemma} \label{Ai}
The inequality
\begin{equation}
\label{5}
\Big|\int_{-H^{\alpha}/4}^{H^{\alpha}/4}w_j(\zeta)w_k(\zeta)d\zeta-\delta_{j,k}\Big|\leq c_N H^{-\alpha}\,,
\end{equation}
is valid for any $j,k=1,\ldots, N$ and with some factor $c_N$ depending on $N\in\mathbb{N}$ and $\alpha=\alpha^\lozenge$.
\end{lemma}

In the same way as in \cite{Freitas, na539} the differential equation \eqref{4} can be derived
by using the coordinate \eqref{zetaH}, inserting the asymptotic ans\"atze \eqref{phiH}, \eqref{LamH}
into the original problem \eqref{DP}, and collecting the terms of order $H^{-5\alpha/4}$.
To make such formal asymptotic analysis rigorous, it is necessary to estimate discrepancies generated
in the Helmholtz equation in $Q^H$. Moreover, aiming to employ the max-min principle, we have
to evaluate reciprocal scalar products of functions \eqref{phiH}. We will do these step by step and, first of all,
observe that the last factor $U^H(x)$ in \eqref{phiH} enjoys the relations
\begin{equation}
\label{UHPr}
\int_{\Pi_z}|U^H(x)|^2\,dy=h(z)^{2},\quad
\int_{\Pi_z}|\nabla_{y}U^H(x)|^2\,dy=\Lambda_\Pi.
\end{equation}

\begin{lemma}
\label{AlOrt}
For any $N\in\mathbb{N}$, there exists $c_N$ such that
\begin{equation*}
\Big|(\Phi_j^H,\Phi_k^H)_H-\delta_{j,k}\Big|\leq c_N H^{-\alpha^\lozenge}, \quad j,k=1,\ldots, N.
\end{equation*}
In other words, the functions $\Phi_1^H, \ldots, \Phi_N^H$ are ``almost orthonormal'' in $L^2(Q^H)$
for a sufficiently big $H$, in particular, they are linear independent.
\end{lemma}
\begin{proof}
Using \eqref{UHPr}, we calculate the scalar product
\begin{equation*}
(\Phi_j^H,\Phi_k^H)_H=
\int_{-H^{2/3}/2}^{H^{2/3}/2}\chi^H(\zeta)^2w_j(\zeta)w_k(\zeta)h(H^{-1/3}\zeta)^{2}d\zeta\,.
\end{equation*}
Since $h(H^{-1/3}\zeta)^{2}=1-4H^{-2/3}|\zeta|+4H^{-4/3}|\zeta|^2$,
according to \eqref{5} and \eqref{chi} it suffices to observe that the integral
$$
\int_{-H^{2/3}/2}^{H^{2/3}/2} \chi^H(\zeta)^2 (|\zeta|+|\zeta|^2)|w_j(\zeta)||w_k(\zeta)|d\zeta
$$
converges due to the exponential decay of $w_j$ and $w_k$.
\end{proof}

\begin{lemma} \label{GrAlmOrt}
There exist positive $H_{N}$ and $c_{N}$ such that
for $H>H_{N}$ we have
\begin{equation} \label{Grad}
\Big|(\nabla_x\Phi_j^H,\nabla_x\Phi_k^H)_H-\Lambda_\Pi\delta_{j,k}\Big|<
c_{N}H^{-\alpha^\lozenge}, \quad j,k=1,\ldots, N.
\end{equation}
\end{lemma}

\begin{proof}
The scalar product on the left of \eqref{Grad}
is the sum of two expressions
\begin{equation*}
I_1=({\nabla_{y}}\Phi_j^H,{\nabla_{y}}\Phi_k^H)_H\quad\text{and}\quad
I_2=(\partial_z\Phi_j^H,\partial_z\Phi_k^H)_H\,.
\end{equation*}
The second equality in \eqref{UHPr} leads to the formula
\begin{equation*}
I_1
=\Lambda_\Pi\int_{-H^{2/3}/2}^{H^{2/3}/2}|\chi^H(\zeta)|^2 w_j(\zeta)w_k(\zeta)d\zeta.
\end{equation*}
Thus, Lemma \ref{Ai} provides the estimate
$|I_1-\Lambda_\Pi\delta_{j,k}|\leq c_N H^{-2/3}$.

In order to evaluate the scalar product $I_2$ we rewrite it as the sum of four expressions
\[
I_2 =
J_1^{jk}+J_2^{jk}+J_3^{jk}+J_3^{kj},\]
where
\[J_1^{jk}=H^{-1/3}\left(\partial_z(\chi^H w_j)U^H;\partial_z(\chi^H w_k)U^H\right)_H,\]
\[J_2^{jk}=H^{-1/3}\left( \chi^H w_j\,\partial_zU^H;\chi^H w_k\,\partial_z U^H\right)_H,\]
\[J_3^{jk}=H^{-1/3}\left(\partial_z(\chi^H w_j)U^H;\chi^H w_k\,\partial_zU^H\right)_H.\]

The first equality \eqref{UHPr} yields
\begin{equation*}
J_1^{jk}=H^{-2/3}\int_{-H^{2/3}/2}^{H^{2/3}/2}
\partial_\zeta\left(\chi^H(\zeta)w_j(\zeta)\right)\,
{\partial_\zeta}\left(\chi^H(\zeta)w_k(\zeta)\right) |h(H^{1/3} \zeta)|^{2}\,d\zeta.
\end{equation*}
Since the factor $|h(H^{1/3} \zeta)|^{2}$ in the integrand is uniformly bounded for $\zeta\in[-H^{2/3}/2,H^{2/3}/2]$,
it is sufficient to estimate the $L^2(\mathbb{R})$-norm of $\partial_\zeta\left(\chi^H w_j\right)$ for $j=1,\ldots, N$.
The solution $w_j$ of the differential equation \eqref{4} satisfies
$$
\|\partial_\zeta w_j;L^2(\mathbb{R})\|^2 \leq \mu_j\|w_j;L^2(\mathbb{R})\|^2<\mu_N\,.
$$
Combining the evident inequality
$|\partial_\zeta \chi^H(\zeta)|\leq c H^{-2/3}$ with the
normalization condition \eqref{4.5}, we conclude that $J^{jk}_{1}\leq c_N H^{-2/3}$.

Furthermore, we have
\begin{equation}
\label{J2}
J_{2}^{jk}=
H^{-1/3}\int_{-H/2}^{H/2}|\chi^H(\zeta)|^2 w_j(\zeta)w_k(\zeta)
|\partial_z (h(z)^{-1})|^2\int_{\Pi_z}
\left|y  \nabla_y U_\Pi(h(z)^{-1}y)\right|^2\,dydz.
\end{equation}
The Fourier analysis of the problem \eqref{DPi} ensures that the eigenfunction $U_\Pi$
and its derivatives decay at infinity at the exponential rate $e^{-\sqrt{\pi^2-\Lambda_\Pi}\,|y|}$.
Thus, the last integral in \eqref{J2}
is bounded and we write
\[J_{2}^{jk}\leq C\int_{-H^{2/3}/2}^{H^{2/3}/2}|\chi^H(\zeta)|^2w_j(\zeta)w_k(\zeta)
|\partial_z h(H^{1/3}\zeta)|^2 |h(H^{1/3}\zeta)|^{-2}\,d\zeta \,.\]
Recalling formula for $h(\tau)$ in \eqref{PiT}, we repeat the above argumentation and derive
the inequality $J_2^{jk}\leq c_N H^{-2}$.

We complete the proof by estimating the scalar product $J_3^{jk}$. The bound $c_N H^{-2/3}$
follows from the Cauchy--Swartz inequality:
\[J_{3}^{jk}\leq
\left(J_{1}^{jj}J_{2}^{kk}\right)^{1/2}\leq c_N H^{-2/3}.
\]
\end{proof}

Eigenvalues of the problem \eqref{DP} can be determined by the max-min principle
(see, e.g, \cite[Th. XIII,1]{RS4})
\begin{equation}
\label{mm}
\lambda_  n^H=\sup_{E}\inf_{u\in E\setminus\{0\}}\frac{\|\nabla u;L^2(Q^H)\|^2}{\|u;L^2(Q^H)\|^2},
\end{equation}
where supremum is calculated over all subspaces $E\subset H^1_0(Q_R)$  of co-dimension $n-1$.
Due to Lemma \ref{AlOrt} the functions $\Phi_j^H$, $j=1,\ldots, N$, are linearly independent
for a large $H$, each subspace $E$ in \eqref{mm} contains a linear combination
$\sum_{j=1}^N\alpha_j\Phi^H_j$ with coefficients subject to $\sum_{j=1}^N|\alpha_j|^2=1$. 
In view of Lemmas \ref{AlOrt} and \ref{GrAlmOrt} we
obtain that for $H>H_N$
\begin{multline*}
\frac{\left\|\sum_{j=1}^N\alpha_j\nabla_x \Phi_j^H;L^2(Q^H)\right\|^2}
{\left\|\sum_{j=1}^N\alpha_j\Phi_j^H;L^2(Q^H)\right\|^2}\leq
\frac{\sum_{j=1}^N|\alpha_j|^2\left\|\nabla_x \Phi_j^H;L^2(Q^H)\right\|^2+{\bf C}_N H^{-2/3}}
{\sum_{j=1}^N|\alpha_j|^2\left\|\Phi_j^H;L^2(Q^H)\right\|^2-
{\bf C}_N H^{-2/3}}\leq \nonumber\\
\leq
\Lambda_\Pi+C_N^\lozenge H^{-2/3},
\end{multline*}
where $H_N$, ${\bf C}_N$ and $C_N^\lozenge$ are some positive constants.

We also choose a bound $H_N^\lozenge$ such that, for $H>H_N^\lozenge$,
the relation $\Lambda_\Pi+C_N^\lozenge H^{-2/3}<\lambda_\dagger^H$ holds
for the cutoff value \eqref{cutoff}.
This guarantees that $\sigma_d^H\geq N$.
The proof of Theorem \ref{Main} is completed in case ({\bf i}).

\section{Total multiplicity in ellipsoidal case }
To investigate the cruciform quantum waveguide with high ellipsoidal cross-section, see Section 1,
it is enough to repeat word-to-word our consideration in Section 3 and modify the attendant computations to
a very little degree. First of all, this case requires in \eqref{PiT} for 
\[h(z)=\left(1-\frac{4z^2}{H^2}\right)^{1/2}.\]
With the same argument as in \cite{BorFre1,NaPeTa} we replace the limit one-dimensional spectral problem
\eqref{4} by the following one:
\begin{equation*}
-\partial^2_\zeta w(\zeta)+4\Lambda_\Pi\zeta^2 w(\zeta) =\mu w(\zeta), \quad \zeta\in\mathbb{R}.
\end{equation*}
Now its eigenvalues take the very simple form
\[\mu_n=2\Lambda_\Pi^{1/2}(2n+1),\qquad n\geq0\]
while the corresponding normalized in $L^2(\mathbb{R})$ eigenfunctions can be expressed in terms of the Hermite polynomials $\mathcal{H}_n$
\cite{Kamke, AS} as follows:
\begin{equation}
\label{51}
w_n(\zeta)=(2^n n!\sqrt\pi)^{-1/2}(4\Lambda_\Pi^{1/2})^{1/4}e^{-\Lambda_\Pi^{1/2}{\zeta^2}}\mathcal{H}_n((4\Lambda_\Pi)^{1/4}\zeta).
\end{equation}

These explicit formulas and the exponential decay of the eigenfunctions \eqref{51} lead to assertions similar to
Lemmas 3.1, 3.2 and 3.3 with the usual change $\alpha^\lozenge\mapsto \alpha^\circ$. Hence,
the max-min principle \eqref{mm} applies just in the same way and supports Theorem \eqref{Main} in case ({\bf ii}) too.

\bigskip

\section{Final remarks}
In view of the localization effect for eigenfunctions of the Dirichlet problem on the cross-section $\omega^H$
of the cylinders \eqref{WGD} composing the waveguide $Q$, the shape of $\omega^H$ can be perturbed outside a neighbourhood
of the mid-line $\{(y,z): z=0, |y|<1\}$, cf. Fig. 4. Let $\omega^H_\bigstar$ be a perturbed cross-section while
$\omega^\bigstar=\{(y,z): |z|<1, |y|<h_\bigstar(z)\}$ is its prototype.

\begin{figure}[ht]
\begin{center}
\begin{minipage}[ht]{0.2\linewidth}
\includegraphics[width=0.85\linewidth]{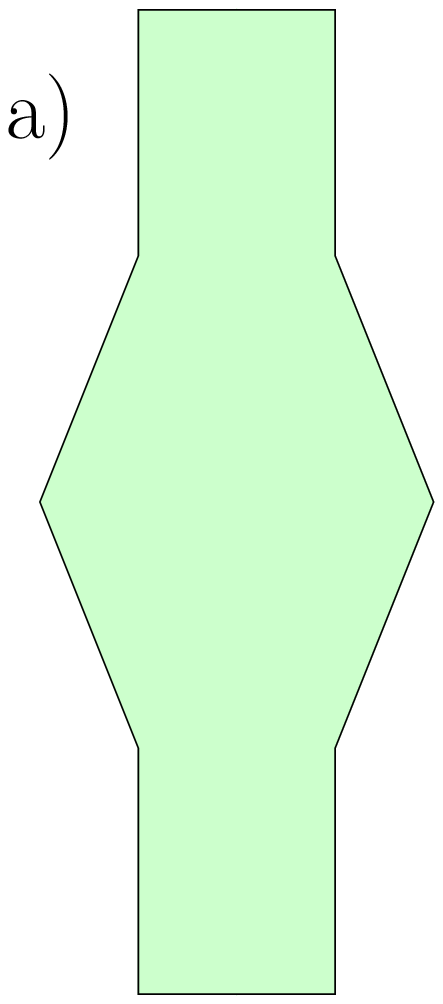}

\end{minipage}
\begin{minipage}[ht]{0.2\linewidth}
\includegraphics[width=0.85\linewidth]{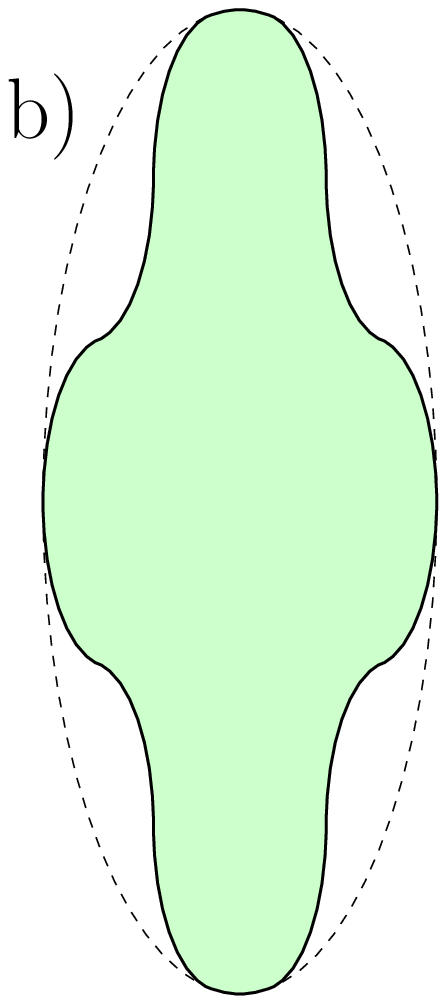}

\end{minipage}
\begin{minipage}[ht]{0.2\linewidth}
\includegraphics[width=0.85\linewidth]{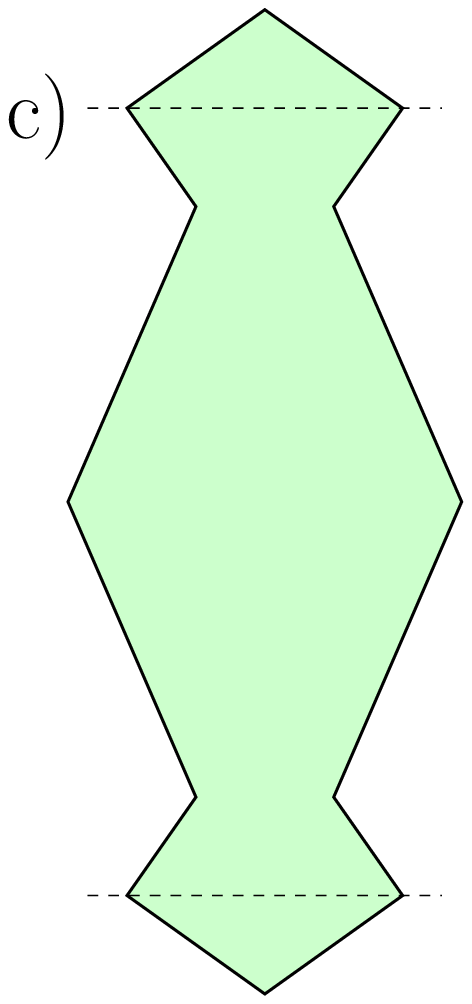}
\end{minipage}
\begin{minipage}[ht]{0.2\linewidth}
\includegraphics[width=0.85\linewidth]{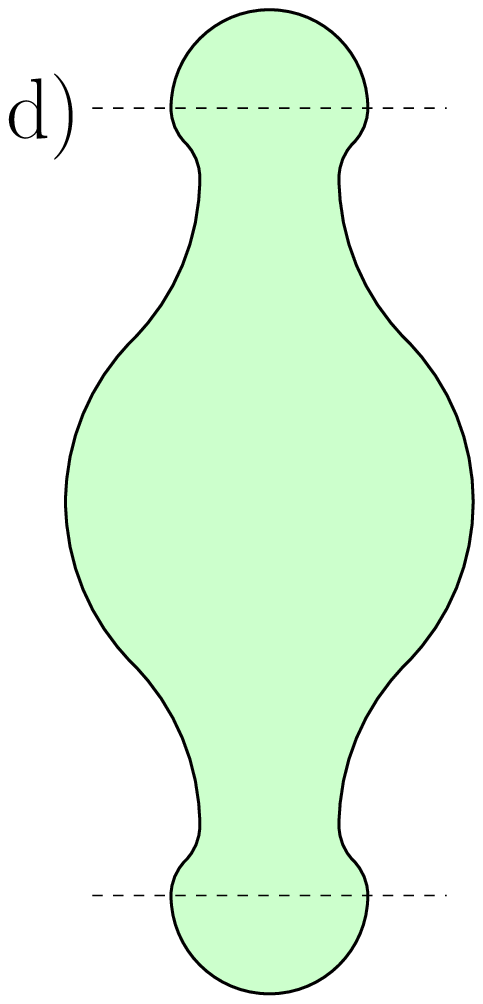}
\end{minipage}
\end{center}
\caption{....}
\end{figure}

In the case $h_\bigstar(z)=h(z)$ for $|z|<\delta$, $0<h_\bigstar(z)<1$ for $|z|\in[\delta,1]$, $\delta>0$ see Fig. 4,a and b,
when the unique global strict maximum of the width function $h_\bigstar$ is attained at the point $z=0$, Theorem \ref{Main}
remains valid even literally.

If the function $h_\bigstar$ has local maxima as depicted in Fig. 4, c and d, at the first sight the asymptotic
procedures in use may help to construct approximate eigenfunctions which are localized near those maximum points.
However, the corresponding eigenvalues get asymptotic forms, different from \eqref{LamH}, and climb above
the lower bound \eqref{cutoff} of the continuous spectrum where the max-min principle does not apply
and detection of eigenvalues becomes much more complicated task, cf. \cite{na489,na546}.
At the same time, recalling elegant trick \cite{EvliVa}, we impose the artificial Dirichlet conditions on the horizontal
cross section $\{x\in Q: z=0\}$ and consider the Dirichlet problem \eqref{DP} in the upper
half $Q^+=\{x\in Q: z>0\}$ of the cruciform waveguide Q. According to \cite{na480}, this problem has at least
one eigenvalue $\lambda^+$ below the principle eigenvalue $\lambda_\dagger(Q^+)$ of the continuous spectrum in $Q^+$.
Moreover, the odd extension in the vertical coordinate $z$ of the corresponding eigenfunction $u^+$ gives an eigenfunction
of the original problem in the intact waveguide $Q$.

For the cruciform waveguide composed from the circular cylinders, cf. \cite{BaMaNaAa}, we readily observe that $\lambda^+$ stays above
the threshold $\lambda_\dagger(Q)<\lambda_\dagger(Q^+)$. Indeed, as was shown in \cite{BaMaNaAa}, the discrete spectrum
consists of the only eigenvalue $\lambda^\circ\in (0, \lambda_\dagger(Q))$ while the corresponding eigenfunction
is even in variable $z$. However, for the waveguides $Q^H$ in cases {\bf (i)} and {\bf (ii)} such inference may become wrong
because the reach family of above-constructed eigenfunctions contains ones which are odd in the vertical coordinate $z$.

\subsection*{Acknowledgement}
Authors were supported by the grant 0.38.237.2014 of St.Petersburg
University and grant 15-01-02175 of Russian Foundation of Basic Research.
First two authors were also supported by ``Native towns'', 
a social investment program of PJSC "Gazprom Neft".

\end{document}